\documentclass[10pt]{paper}

\usepackage{amsmath,amsfonts,amsthm,mathtools}
\usepackage{algorithm2e}

\usepackage{pgfplotstable}

\usepackage{microtype}
\usepackage{hyperref}

\newcommand{\m}{\mathbf{m}}
\newcommand{\s}{\mathbf{s}}
\renewcommand{\v}{\mathbf{v}}
\newcommand{\w}{\mathbf{w}}
\newcommand{\x}{\mathbf{x}}
\newcommand{\y}{\mathbf{y}}
\newcommand{\z}{\mathbf{z}}
\newcommand{\ones}{\mathbf{1}}

\DeclarePairedDelimiter{\norm}{\lVert}{\rVert}

\newtheorem{theorem}{Theorem}
\newtheorem{lemma}[theorem]{Lemma}
\newtheorem{corollary}[theorem]{Corollary}
\newtheorem{remark}[theorem]{Remark}

\usepackage{booktabs,colortbl}
\pgfplotstableset{
every even row/.style={
before row={\rowcolor[gray]{0.9}}},
every head row/.style={
before row=\toprule,after row=\midrule},
every last row/.style={
after row=\bottomrule},
}

\begin{document}

%\runningheads{B.~Meini and F.~Poloni}{Perron--based algorithms for the multilinear PageRank}

\title{Perron--based algorithms for the multilinear PageRank}

%\author{Beatrice Meini\affil{1} and Federico Poloni\affil{2}\corrauth}
%\address{\affilnum{1} Dipartimento di Matematica, Universit\`a di Pisa, Italy\break
%\affilnum{2} Dipartimento di Informatica, Universit\`a di Pisa, Italy}

%\corraddr{E-mail: federico.poloni@unipi.it}
\author{Beatrice Meini\footnote{Dipartimento di Matematica, Universit\`a di Pisa, Italy.} and Federico Poloni\footnote{Dipartimento di Informatica, Università di Pisa, Italy. \texttt{federico.poloni@unipi.it}.}}
\maketitle

\begin{abstract}
We consider the multilinear PageRank problem studied in [Gleich, Lim and Yu, \emph{Multilinear PageRank}, 2015], which is a system of quadratic equations with stochasticity and nonnegativity constraints. We use the theory of quadratic vector equations to prove several properties of its solutions and suggest new numerical algorithms. In particular, we prove the existence of a certain minimal solution, which does not always coincide with the stochastic one that is required by the problem.
We use an interpretation of the solution as a Perron eigenvector to devise new fixed-point algorithms for its computation, and pair them with a continuation strategy based on a perturbative approach. The resulting numerical method is more reliable than the existing alternatives, being able to solve a larger number of problems.
\end{abstract}


\maketitle

\section{Introduction}

Gleich, Lim and Yu~\cite{GleLY15} consider the following problem, arising as an approximate computation of the stationary measure of an order-2 Markov chain: given $\v \in \mathbb{R}^n, R \in \mathbb{R}^{n\times n^2}$, $\alpha\in \mathbb{R}$ with $\v\geq 0, R\geq 0, \alpha \in (0,1)$ and 
\begin{equation} \label{stochastic}
	\ones_n^\top \v = 1, \quad \ones_n^\top R = \ones_{n^2}^\top,
\end{equation}
solve for $\x$ the equation
\begin{equation}  \label{mlpr}
	\x = \alpha R(\x \otimes \x) + (1-\alpha) \v.
\end{equation}
The solution of interest $\s$ is stochastic, i.e., $\s\ge 0$  and $\ones_n^\top \s = 1$. Here $\ones$ denotes a column vector of all ones (with an optional subscript to specify its length), and inequalities between vectors and matrices are intended in the componentwise sense.
In the paper~\cite{GleLY15}, they prove some theoretical properties, consider several solution algorithms, and evaluate their performance. In the more recent paper~\cite{LiLNV17}, the authors improve some results concerning  the uniqueness of the solution.

This problem originally appeared in~\cite{LiN14}, and is a variation of problems related to tensor eigenvalue problems and Perron--Frobenius theory for tensors; see, e.g.,~\cite{Lim05,Cha08,Fri13}. However, it also fits in the framework of quadratic vector equations derived from Markovian binary tree models introduced in~\cite{HauLR08} and later considered in~\cite{BinMP11,MeiP11,Pol13}.
Indeed, the paper~\cite{Pol13} considers a more general problem, which is essentially~\eqref{mlpr} without the hypotheses~\eqref{stochastic}. Hence, all of its results apply here, and they can be used in the context of multilinear PageRank. In particular, \cite{Pol13} considers the minimal nonnegative solution of~\eqref{mlpr} (in the componentwise sense), which is not necessarily stochastic as the one sought in~\cite{GleLY15}. 

In this paper, we use the theory of quadratic vector equations in~\cite{BinMP11,MeiP11,Pol13} to better understand the behavior of the solutions of \eqref{mlpr}
and suggest new algorithms for computing the stochastic solution.
More specifically, we show that if one considers the minimal nonnegative solution of \eqref{mlpr} as well, the theoretical properties of~\eqref{mlpr} become clearer, even if one is only interested in stochastic solutions.
Indeed we prove that there always exists a minimal nonnegative solution, which is the unique stochastic solution  when $\alpha\le 1/2$. 	When $\alpha > \frac12$,  the minimal nonnegative solution $\m$ is not stochastic  and there is at least one stochastic solution $\s \geq \m$. 
Note that~\cite[Theorem~4.3]{GleLY15} already proves that when $\alpha \leq \frac12$ the  stochastic solution is unique and~\cite[Theorem~1]{LiLNV17} slightly improves this bound; our results give a broader characterization. 
All this is in Section~\ref{sec:properties}.

When $\alpha\le 1/2$, as already pointed out in \cite{GleLY15}, computing the stochastic solution of \eqref{mlpr} is easy. Indeed, this is also due to the fact that the stochastic solution is the minimal solution, and for instance the numerical methods proposed in \cite{HauLR08,BinMP11,MeiP11} perform very well.
The most difficult case is when $\alpha>1/2$, in particular when $\alpha\approx 1$. Since the minimal solution $\m$ of~\eqref{mlpr}  can be easily computed, the idea is to compute and deflate it, with a similar strategy to the one developed in~\cite{BinMP11,MeiP11},   hence allowing us to compute stochastic solutions even when they are not minimal. 
The main tool in this approach is rearranging~\eqref{mlpr} to show that (after a change of variables) a solution $\x$ corresponds to the Perron eigenvector of a certain matrix that depends on $\x$ itself. This interpretation in terms of Perron vector allows to devise new algorithms based on fixed point iteration and on Newton's method.
Sections~\ref{sec:deflating} and~\ref{sec:perron} describe this deflation technique and the algorithms based on the Perron vector computation.

Finally, we propose in Section~\ref{sec:homotopy} a continuation strategy based on a perturbative approach that allows one to solve the problem for values $\hat{\alpha} < \alpha$ in order to obtain better starting values for the more challenging cases when $\alpha \approx 1$.

We report several numerical experiments in Section~\ref{sec:experiments}, to show the effectiveness of these new techniques for the set of small-scale benchmark problems introduced in~\cite{GleLY15}, and draw some conclusions in Section~\ref{sec:conclusions}.

\section{Properties of the nonnegative solutions} \label{sec:properties}

In this section, we show properties of the nonnegative solutions of the equation \eqref{mlpr}. In particular, we prove that there always exists a minimal nonnegative solution, which is stochastic when $\alpha\le 1/2$.
These properties can be derived by specializing the results of \cite{Pol13}, which apply to more general vector equations defined by bilinear forms.

We introduce the map 
\[
G(\x) : = \alpha R(\x \otimes \x) + (1-\alpha) \v,
\]
and its Jacobian 
\[
G'_{\x} := \alpha R(\x \otimes I_n) + \alpha R(I_n \otimes \x).
\]

We have the following result.
\begin{lemma}
	Consider the fixed-point iteration
	\begin{equation} \label{fp}
		\x_{k+1} = G(\x_k), \quad k=0,1,\dots,
	\end{equation}
	started from $\x_0=0$. Then the sequence of vectors 
	$\{\x_k\}$ is such that $0\le \x_k\le \x_{k+1}\le \ones$, there exists $\lim_{k\to\infty}\x_k=\m$ and 
$\m$ is the minimal nonnegative solution of \eqref{mlpr}, i.e., equation \eqref{mlpr} has a (unique) solution $\m\geq 0$ such that $\m \leq \x$ for any other possible solution $\x\geq 0$. 
% 	Similarly, the Newton method applied to the function $F(\x)=\x-G(\x)$, namely
% with $\x_0=0$, is well defined and converges monotonically to $\m$.\label{lem:min}
\end{lemma}

\begin{proof}
The map $G(\x)$ is weakly positive, i.e., $G(\x)\ge 0$, $G(\x)\ne 0$ whenever $\x\ge0$,   $\x\ne0$. Moreover, if $0\le \x\le\ones$ then $0\le G(\x)\le\ones$. Therefore Condition A1 of \cite{Pol13} is satisfied which, according to Theorem 4 of \cite{Pol13},  implies that the sequence of vectors 
	$\{\x_k\}$ is bounded and converges monotonically to a vector $\m$, which is the minimal nonnegative solution of \eqref{mlpr}.
% Newton's method applied to the function $F(\x)$ generates the sequence
% $ (I-G'_{\x_k})(\x_{k+1}-\x_k) = -F(\x_k)$. By using the expression of $G'_{\x}$ and of $F(\x)$ we easily arrive to \eqref{eq:newton}.
% Concerning applicability and convergence, observe that $G'_{\m}$ is a nonnegative matrix such that $\ones_n^\top G'_{\m}= \alpha \ones_n^\top\left (R(\m \otimes I_n) +  R(I_n \otimes \m)\right)\le 2\alpha (\ones_n^\top \m) \ones_n^\top<\ones_n^\top$. (ATTENZIONE: SE SI USA QUESTO ARGOMENTO BISOGNA AVER GIA' PARLATO DI CASO SUPER-SUBCRITICAL) Hence the spectral radius of $G'_{\m}$ is less than 1 and $I-G'_{\m}$ is a nonsingular M-matrix.
% In particular, for any $0\le \x\le \m$, one has $0\le G'(\x)\le G'_{\m}$, therefore the matrix $I-G'_{\x}$ is a nonsingular M-matrix.
% By applying the same arguments as in \cite{Pol13}, we may show that the Newton method is well defined and monotonically convergent when $\x_0=0$.
\end{proof}
% Algorithm~\eqref{algo:newton} implements this Newton-based method.
% \begin{algorithm}
% \KwIn{$R$, $\alpha$, $v$ as above, a tolerance $\varepsilon$}
% \KwOut{An approximation to the minimal solution $\m$ of~\eqref{mlpr}.}
% $\x \leftarrow 0$\;
% \While{$\norm{G(\x) - \x}_1 > \varepsilon$ }{
% 	$G'_{\x} \leftarrow  \alpha R(\x \otimes I_n) + \alpha R(I_n \otimes \x)$\;
% 	$\x \leftarrow \x - (I_n - G'_{\x})^{-1}((1-\alpha) \v - \alpha R(\x \otimes \x)) $
% }
% $\m = \x$\;
% \caption{Newton's method for the computation of the minimal solution $\m$ to~\eqref{mlpr}.} \label{algo:newton}
% \end{algorithm}

\subsection{Sum of entries and criticality}
In this specific problem, the hypotheses~\eqref{stochastic} enforce a stronger structure on the iterates of~\eqref{fp}: the sum of the entries of $G(\x)$ is a function of the sum of the entries of $\x$ only.

\begin{lemma}
Let $g(u):=\alpha u^2 + (1-\alpha)$. Then,  $\ones^\top G(\x) = g(\ones^\top \x)$ for any $\x\in\mathbb R^n$.\label{lem:gu}
\end{lemma}
\begin{proof}
\begin{align*}
\ones^\top G(\x) &= \ones^\top (\alpha R(\x \otimes \x) + (1-\alpha) \v) = \alpha \ones^\top R(\x \otimes \x) + (1-\alpha) \ones^\top \v 
\\&= \alpha \ones^\top(\x \otimes \x) + (1-\alpha) = \alpha (\ones^\top \x)^2 + (1-\alpha). \qedhere
\end{align*}
\end{proof}
This fact has important consequences for the sum of the entries of the solutions of~\eqref{mlpr}.
\begin{corollary}
For \label{solvalues} each solution $\x$ of~\eqref{mlpr}, $\ones^\top \x$ is one of the two solutions of the quadratic $u = g(u)$, i.e., $u=1$ or $u=\frac{1-\alpha}{\alpha}$.
\end{corollary}

Let $u$ be one of the solutions of $u = g(u)$ and define the level set 
$\ell_u = \{\x: \ones^\top\x=u, \x\geq 0\}$. Since $\ell_u$ is convex and compact, and since $G(\x)$ maps $\ell_u$ to itself by Lemma~\ref{lem:gu}, then  the Brouwer fixed-point theorem implies the following result.

\begin{corollary}
There exists at least a  solution $\x\ge 0$ to~\eqref{mlpr} with $\ones^\top\x = 1$ and a solution $\x\ge0$ with $\ones^\top\x = \frac{1-\alpha}{\alpha}$. 
\end{corollary}

Hence we can have two different settings, for which we borrow the terminology from~\cite{Pol13}.
\begin{description}
	\item[Subcritical case] $\alpha \leq \frac12$, hence the minimal nonnegative solution $\m = \s$ is the unique stochastic solution.
	\item[Supercritical case] $\alpha > \frac12$, hence the minimal nonnegative solution $\m$ satisfies $\ones^\top\m = \frac{1-\alpha}{\alpha} < 1$ and there is at least one stochastic solution $\s \geq \m$. 
\end{description}
Note that~\cite[Theorem~4.3]{GleLY15} already proves that when $\alpha \leq \frac12$ the  stochastic solution is unique; these results give a broader characterization.

The tools that we have introduced can already be used to determine the behavior of simple iterations such as~\eqref{fp}.

\begin{theorem}
	Consider the fixed-point iteration~\eqref{fp}, with a certain initial value $\x_0\ge 0$, for the problem~\eqref{mlpr} with $\alpha > \frac12$. Define $z_k := \ones^\top \x_k$. Then,
\begin{itemize}
 	\item If $z_0 \in (0, \frac{1-\alpha}{\alpha}]$, then $\lim_{k\to\infty} z_k = \frac{1-\alpha}{\alpha}$, and the iteration~\eqref{fp} can converge only to the minimal solution $\m$ (if it converges).
 	\item If $z_0 \in (\frac{1-\alpha}{\alpha}, 1]$, then $\lim_{k\to\infty} z_k = 1$, hence the iteration~\eqref{fp} can converge only to a stochastic solution (if it converges).
 	\item If $z_0 \in (1, +\infty)$, then $\lim_{k\to\infty} z_k = +\infty$, hence the iteration~\eqref{fp} diverges.
 \end{itemize} 
\end{theorem}
\begin{proof}
Thanks to Lemma~\ref{lem:gu}, the quantity $z_k := \ones^\top \x_k$ evolves according to $z_{k+1} = g(z_k)$. So the result follows by the theory of scalar fixed-point iterations, since this iteration converges to $\frac{1-\alpha}{\alpha}$ for $z_0 \in (0, \frac{1-\alpha}{\alpha}]$, to $1$ for $z_0 \in (\frac{1-\alpha}{\alpha}, 1]$, and diverges for $z_0 \in (1, +\infty)$.
\end{proof}
An analogous result holds for the subcritical case.

The papers~\cite{BinMP11,MeiP11,Pol13} describe several methods to compute the minimal solution $\m$. In particular, all the ones described in~\cite{Pol13} exhibit monotonic convergence, that is, $0 = \x_0 \leq \x_1 \leq \x_2 \leq \cdots \leq \x_k \leq \cdots \leq \m$. Due to the uniqueness and the monotonic convergence properties, computing the minimal solution $\m$ is typically simple, fast, and free of numerical issues. Hence in the subcritical case the multilinear PageRank problem is easy to solve. The supercritical case is more problematic.

Among all available algorithms to compute the minimal solution $\m$, we recall Newton's method, which is one of the most efficient ones. The Newton--Raphson method applied to the function $F(\x) = \x - G(\x)$ generates the sequence
	\begin{equation}\label{eq:newton}
		(I-G'_{\x_k})\x_{k+1} = (1-\alpha) \v - \alpha R(\x_k \otimes \x_k), \quad k=0,1,\dots.
	\end{equation}
The following result holds~\cite[Theorem~13]{Pol13}.
\begin{lemma}
	Suppose that $\m > 0$, and that $G'_{\m}$ is irreducible. Then, Newton's method~\eqref{eq:newton} starting from $\x_0=0$ is well defined and converges monotonically to $\m$ (i.e., $0 = \x_0 \leq \x_1 \leq \x_2 \leq \cdots \leq \x_k \leq \cdots \leq \m$).
\end{lemma}
Algorithm~\eqref{algo:newton} shows a straightforward implementation of Newton's method as described above.
\begin{algorithm}
\KwIn{$R$, $\alpha$, $\v$ as above, a tolerance $\varepsilon$.}
\KwOut{An approximation to the minimal solution $\m$ of~\eqref{mlpr}.}
$\x \leftarrow 0$\;
\While{$\norm{G(\x) - \x}_1 > \varepsilon$ }{
	$G'_{\x} \leftarrow  \alpha R(\x \otimes I_n) + \alpha R(I_n \otimes \x)$\;
	$\x \leftarrow (I_n - G'_{\x})^{-1}((1-\alpha) \v - \alpha R(\x \otimes \x)) $\;
}
$\m = \x$\;
\caption{Newton's method for the computation of the minimal solution $\m$ to~\eqref{mlpr}.} \label{algo:newton}
\end{algorithm}

Note that the theory in~\cite[Section~9]{Pol13} shows how one can predict where zeros appear in $\m$ and eliminate them reducing the problem to a smaller one. Indeed, in view of the probabilistic interpretation of multilinear PageRank, zero entries in $\m$ can appear only when the second-order Markov chain associated to $R$ is not irreducible. So we can assume in the following that $\m > 0$, and that the nonnegative matrix $G'_{\m}$ is irreducible. In particular, in this case $G'_{\m}$ also substochastic (as proved in~\cite[Theorem~6]{Pol13}).

\section{Deflating the minimal solution} \label{sec:deflating}
From now on, we assume to be in the supercritical case, i.e., $\alpha>1/2$, and that $\m$ has been already computed and is explicitly available. 

We wish adapt to this setting the deflation strategy introduced in~\cite{MeiP11}. Since all solutions $\s$ to~\eqref{mlpr} satisfy $\s \geq \m$, it makes sense to change variable to obtain an equation in $\y := \s - \m \geq 0$. After a few manipulations, using bilinearity of $(\m+\y)\otimes (\m+\y)$ and the fact that $\m = \alpha R(\m \otimes \m) + (1-\alpha)\v$, one gets
\begin{equation} \label{optimistic}
	\y = \alpha R(\y \otimes \y) + \alpha R(\y \otimes \m) + \alpha R(\m \otimes \y) = \alpha R(\y \otimes \y) + G'_{\m}\y = (\alpha R(\y \otimes I_n) + G'_{\m})\y.
\end{equation}
Moreover,
\begin{equation} \label{normy}
	\ones^\top \y = \ones^\top (\s-\m) = 1-\frac{1-\alpha}{\alpha} = \frac{2\alpha-1}{\alpha}.	
\end{equation}

We set $P_\y := \alpha R(\y \otimes I_n) + G'_{\m}$. Note that $P_\y \geq G'_{\m}$ for each $\y\geq 0$, hence it is irreducible. In addition, if $\y$ is chosen such that $\ones^\top \y = \frac{2\alpha-1}{\alpha}$ as in~\eqref{normy}, then
\begin{equation} \label{conti}
	\begin{aligned}
		\ones^\top P_\y &= \alpha \ones^\top R(\y \otimes I) + \alpha \ones^\top R(I \otimes \m) + \alpha \ones^\top R(\m \otimes I)
		\\&= \alpha \ones^\top (\y \otimes I) + \alpha \ones^\top (I \otimes \m) + \alpha \ones^\top (\m \otimes I)
		\\&= \alpha \frac{2\alpha-1}{\alpha} \ones^\top + \alpha\frac{1-\alpha}{\alpha} \ones^\top + \alpha\frac{1-\alpha}{\alpha} \ones^\top = \ones^\top,
	\end{aligned}	
\end{equation}
so $P_\y$ is a stochastic matrix.

Let us introduce the map $\mathcal{PV}(A)$ that gives the Perron vector $\w$ of an irreducible matrix $A \geq 0$, normalized such that $\ones^\top \w = 1$. 
Then, since $P_\y$ is irreducible and stochastic, \eqref{optimistic} and~\eqref{normy} show that
\begin{equation} \label{optimisticperron}
	\y = \frac{2\alpha-1}{\alpha} \mathcal{PV}(P_\y),
\end{equation}
i.e., the sought vector $\y$ is the Perron vector of the matrix $P_\y$, multiplied by the constant $ \frac{2\alpha-1}{\alpha}$.

\section{Perron-based algorithms} \label{sec:perron}

Equation \eqref{optimisticperron} suggests a new fixed-point iteration for computing $\y$, 
which is analogous to the one appearing in~\cite{MeiP11},
\begin{equation}
	\y_{k+1} = \frac{2\alpha-1}{\alpha} \mathcal{PV}(P_{\y_k}),
\end{equation}
starting from a given nonnegative vector $\y_0$ such that $\ones^\top \y_0 = \frac{2\alpha-1}{\alpha}$. This iteration is implemented in Algorithm~\ref{algo:perron}

\begin{algorithm}
\KwIn{$R$, $\alpha$, $\v$ as above, with $\alpha > \frac12$, a tolerance $\varepsilon$, an initial value $\x_0$.}
\KwOut{An approximation to a stochastic solution $\s$ of~\eqref{mlpr}.}
Compute $\m$ with Algorithm~\ref{algo:newton}\;
Normalize $\x_0$ (if needed): $\x_0 \leftarrow \max(\x_0,\mathbf{0})$, $\x_0 \leftarrow \frac{\x_0}{\ones^\top \x_0}$\;
$\y \leftarrow \x_0 - \m$\;
Normalize $\y$ (if needed): $\y \leftarrow \max(\y,\mathbf{0})$, $\y \leftarrow \frac{2\alpha-1}{\alpha} \frac{\y}{\ones^\top \y}$\;
$G'_{\m} \leftarrow  \alpha R(\m \otimes I_n) + \alpha R(I_n \otimes \m)$\;
\While{$\norm{G(\x) - \x}_1 > \varepsilon$ }{
	$P_\y \leftarrow \alpha R(\y \otimes I_n) + G'_{\m}$\;
	$\y \leftarrow \frac{2\alpha-1}{\alpha} \mathcal{PV}(P_\y)$\;
	$\x \leftarrow \m + \y$\;
}
$\s = \x$\;
\caption{The Perron-based iteration for the computation of a stochastic solution $\s$ to~\eqref{mlpr}.} \label{algo:perron}
\end{algorithm}

%\subsection{Perron-Newton method}

We may also apply Newton's method to find a solution to~\eqref{optimisticperron}, following~\cite{BinMP11}. 
To this end, we first compute the Jacobian of the function $\w(\y):=\mathcal{PV}(P_\y)$.

\begin{lemma}
Let $\w(\y):= \mathcal{PV}(P_\y)$, with $\y\ge0$ such that $\ones^\top \y = \frac{2\alpha-1}{\alpha}$. Then, its Jacobian is given by
\begin{equation} \label{Jac}
	\w'(\y) = \alpha \left((I-P_\y + \w(\y)\ones^\top)^{-1}R(I\otimes \w(\y)) - \w(\y) \ones^\top\right).
\end{equation}
\end{lemma}
\begin{proof}
Let us compute its directional derivative along the direction $\z$. We set $\y(h) = \y+h\z$; hence, $\frac{\mathrm{d}}{\mathrm{d}h} P_{\y(h)} = \alpha R (\z \otimes I)$. Since $P_{\y}$ is irreducible, its Perron eigenvalue is a simple eigenvalue, and hence we can define locally smooth functions $\lambda(h)$ as the Perron eigenvalue of $P_{\y(h)}$ and $\v(h)=\mathcal{PV}(P_{\y(h)})$.
A computation analogous to~\eqref{conti} shows that $\ones^\top P_{\y(h)} = (1+ h \alpha \ones^\top\z)\ones^\top$, hence $\lambda(h) = 1+ h \alpha \ones^\top\z$ and $\frac{\mathrm{d}}{\mathrm{d}h} \lambda(h) = \alpha (\ones^\top \z)$.

By differentiating $P_{\y(h)}\v(h) = \lambda(h)\v(h)$ and evaluating at $h=0$, we get
 \[
\alpha R (\z \otimes I)\v(0) + P_\y \frac{\mathrm{d}\v}{\mathrm{d}h}(0) = \alpha (\ones^\top \z) \v(0) + \frac{\mathrm{d}\v}{\mathrm{d}h} (0),
 \]
or, rearranging terms,
\[
(I-P_\y)\frac{\mathrm{d}\v}{\mathrm{d}h}(0) = \alpha \left(R(I\otimes \v(0)) - \v(0) \ones^\top\right)\z,
\]
where $\v(0)=\w(\y)$.
Since we defined $\v(h)$ so that $\ones^\top \v(h)=1$, we have $\ones^\top \frac{\mathrm{d}\v}{\mathrm{d}h}(0)=0$, and hence also
\[
(I-P_\y + \v(0)\ones^\top)\frac{\mathrm{d}\v}{\mathrm{d}h}(0) = \left(R(I\otimes \v(0)) - \v(0) \ones^\top\right)\z.
\]
The matrix in the left-hand side is nonsingular, since it can be obtained by replacing the eigenvalue $0$ with $1$ in the eigendecomposition of the singular irreducible M-matrix $I-P_\y$. Thus we can write 
\[
\frac{\mathrm{d}\v}{\mathrm{d}h}(0) = \alpha(I-P_\y + \v(0)\ones^\top)^{-1}\left(R(I\otimes \v(0)) - \v(0) \ones^\top\right)\z.
\]
 Since $(I-P_\y + \v(0)\ones^\top) \v(0) = \v(0)$, we can simplify this expression further to
\[
\frac{\mathrm{d}\v}{\mathrm{d}h}(0) = \alpha\left((I-P_\y + \v(0)\ones^\top)^{-1}R(I\otimes \v(0)) - \v(0) \ones^\top \right)\z,
\]
from which we  get~\eqref{Jac}.
\end{proof}

From the above result, the Jacobian of the function
$H(\y)=\y - \frac{2\alpha-1}{\alpha} \mathcal{PV}(P_\y)$ is given by
\begin{equation} \label{Jacnewt}
	H'_\y = I_n +  (2\alpha-1) \w(\y) \ones^\top - (2\alpha-1)(I-P_\y + \w(\y)\ones^\top)^{-1}R(I\otimes \w(\y))
\end{equation}
and Newton's method consists of generating the sequence of vectors
\[
\y_{k+1}=\y_k-(H'_{\y_k})^{-1} H(\y_k)
\]
The Perron--Newton method is described in Algorithm~\ref{algo:perronnewton}.
\begin{algorithm}
\KwIn{$R$, $\alpha$, $\v$ as above, with $\alpha > \frac12$, a tolerance $\varepsilon$, an initial value $\x_0$.}
\KwOut{An approximation to a stochastic solution $\s$ of~\eqref{mlpr}.}
Compute $\m$ with Algorithm~\ref{algo:newton}\;
Normalize $\x_0$ (if needed): $\x_0 \leftarrow \max(\x_0,\mathbf{0})$, $\x_0 \leftarrow \frac{\x_0}{\ones^\top \x_0}$\;
$\y \leftarrow \x_0 - \m$\;
Normalize $\y$ (if needed): $\y \leftarrow \max(\y,\mathbf{0})$, $\y \leftarrow \frac{2\alpha-1}{\alpha} \frac{\y}{\ones^\top \y}$\;
$G'_{\m} \leftarrow  \alpha R(\m \otimes I_n) + \alpha R(I_n \otimes \m)$\;
\While{$\norm{G(\x) - \x}_1 > \varepsilon$ }{
	$P_\y \leftarrow \alpha R(\y \otimes I_n) + G'_{\m}$\;
	$\w \leftarrow \mathcal{PV}(P_\y)$\;
	$H'_\y \leftarrow I_n +  (2\alpha-1) \w \ones^\top - (2\alpha-1)(I-P_\y + \w\ones^\top)^{-1}R(I\otimes \w)$\;
	$\y \leftarrow \y - {H'_\y}^{-1} (\y - \frac{2\alpha-1}{\alpha}\w)$\;

	Normalize $\y$ (if needed): $\y \leftarrow \frac{2\alpha-1}{\alpha} \frac{\y}{\ones^\top \y}$\;
	$\x \leftarrow \m + \y$\;
}
$\s = \x$\;
\caption{The Perron--Newton method for the computation of a stochastic solution $\s$ to~\eqref{mlpr}.} \label{algo:perronnewton}
\end{algorithm}

The standard theorems~\cite{OrtegaBook} on local convergence of Newton's method imply the following result.
\begin{theorem}
	Suppose that the matrix $H'_{\s - \m}$ is nonsingular. Then the Perron--Newton method is locally convergent to a stochastic solution $\s$ of~\eqref{mlpr}. 
\end{theorem}
\begin{remark}
Since $\ones^\top \left(
\w(\y) \ones^\top - (I-P_\y + \w(\y)\ones^\top)^{-1}R(I\otimes \w(\y))
\right)=0$,
the matrix $H'_{\y}$ has an eigenvalue 1 with left eigenvector $\ones^\top$, for each value of $\y$.
\end{remark}

\section{Continuation techniques} \label{sec:homotopy}
The above algorithms, as well as those in~\cite{GleLY15}, are sufficient to solve most of the test problems that are explored in~\cite{GleLY15}. However, especially when $\alpha\approx 1$, the algorithms may converge very slowly or stagnate far away from the minimal solution. For this reason, we explore an additional technique based on a perturbation expansion of the solution as a function of $\alpha$, inspired loosely by the ideas of homotopy continuation~\cite{Li97}, which is a well-known strategy to derive approximate solutions for parameter-dependent equations.

The main result is the following.
\begin{lemma}
Let $\s$ be a stochastic solution of problem~\eqref{mlpr} (for a certain $\alpha>\frac12$), and suppose that $I - G'_{\s}$ is nonsingular. Then, there is a stochastic solution $\s_{\hat{\alpha}}$ to~\eqref{mlpr} with the parameter $\alpha$ replaced by $\hat{\alpha}$ such that
\begin{equation} \label{homotopy}
	\s_{\hat{\alpha}} = \s + \s^{(1)}(\hat{\alpha}-\alpha) + \s^{(2)}(\hat{\alpha}-\alpha)^2 + O((\hat{\alpha} - \alpha)^3),
\end{equation}
with
$\s^{(1)} = (I - G'_{\alpha,\s_\alpha})^{-1}(R(\s_\alpha \otimes \s_\alpha) - \v)$, $\s^{(2)} = 
(I - G'_{\alpha,\s_\alpha})^{-1} R\left(\mathbf{t}_\alpha \otimes \s^{(1)} + \s^{(1)}\otimes \mathbf{t}_\alpha\right)$, where $\mathbf{t}_\alpha := 2\s_{\alpha} + \alpha \s^{(1)}$.
\end{lemma}
\begin{proof}
	Let us make the dependence of the various quantities on the parameter $\alpha$ explicit, i.e., we write $\s_\alpha$ to denote a stochastic solution of~\eqref{mlpr} for a certain value of the parameter $\alpha$, and similarly $G_\alpha, G'_{\alpha,\x}$ and $F_\alpha$.

	We apply the implicit function theorem~\cite[Theorem~9.28]{BabyRudin} to
	\[
		0 = F_\alpha(\s_\alpha) = \s_\alpha - \alpha R(\s_\alpha \otimes \s_\alpha ) - (1-\alpha) \v,
	\]
	obtaining
	\begin{subequations} \label{firstderivative_all}
	\begin{align}
	\frac{\partial F_\alpha}{\partial \s_\alpha} & = I - \alpha R(s_\alpha \otimes I) - \alpha R(I \otimes s_\alpha) = I - G'_{\alpha,\s_\alpha}, \nonumber\\
	\frac{\partial F_\alpha}{\partial \alpha} &= \v - R(\s_\alpha \otimes \s_\alpha),
	\nonumber\\
	\s^{(1)} = \frac{d}{d\alpha} \s_{\alpha} &= -\left(\frac{\partial F_\alpha}{\partial \s_\alpha}\right)^{-1} \frac{\partial F_\alpha}{\partial \alpha} = -(I - G'_{\alpha,\s_\alpha})^{-1}(\v - R(\s_\alpha \otimes \s_\alpha)). 
	\end{align}
	\end{subequations}
	Differentiating~\eqref{firstderivative_all}, we get
	\begin{align*}
	\frac{d}{d\alpha} \frac{\partial F_\alpha}{\partial \s_\alpha} &= -R(\s_\alpha \otimes I) -R(I\otimes \s_\alpha) - \alpha R \left(\s^{(1)} \otimes I\right) - \alpha R \left(I\otimes \s^{(1)}\right), \\
	\frac{d}{d\alpha} \frac{\partial F_\alpha}{\partial \alpha} &= -R\left(\s^{(1)} \otimes \s_\alpha\right)-R\left(\s_\alpha \otimes \s^{(1)}\right) \\
	\s^{(2)} = \frac{d^2 \s_{\alpha}}{d\alpha^2} &= \left(\frac{\partial F_\alpha}{\partial \s_\alpha}\right)^{-1} \left( \frac{d}{d\alpha} \frac{\partial F_\alpha}{\partial \s_\alpha}\right) \left(\frac{\partial F_\alpha}{\partial \s_\alpha}\right)^{-1} \frac{\partial F_\alpha}{\partial \alpha} - \left(\frac{\partial F_\alpha}{\partial \s_\alpha}\right)^{-1} \frac{d}{d\alpha}\frac{\partial F_\alpha}{\partial \alpha} \\
	&= -\left(\frac{\partial F_\alpha}{\partial \s_\alpha}\right)^{-1} \left( \frac{d}{d\alpha} \frac{\partial F_\alpha}{\partial \s_\alpha}\right) \s^{(1)} - \left(\frac{\partial F_\alpha}{\partial \s_\alpha}\right)^{-1} \frac{d}{d\alpha}\frac{\partial F_\alpha}{\partial \alpha}\\
	&= (I - G'_{\alpha,\s_\alpha})^{-1} \left(R \left(\mathbf{t}_\alpha \otimes \s^{(1)}\right) + R \left(\s^{(1)}\otimes \mathbf{t}_\alpha\right)  \right).
	\end{align*}

	The function $\s_{\hat{\alpha}}$ obtained by the theorem must satisfy $\ones^\top \s_{\hat{\alpha}} = 1$ for each $\hat{\alpha} > \frac12$: indeed, by Corollary~\ref{solvalues} for a solution $\s_{\hat{\alpha}}$ of the equation $\x = G_{\hat{\alpha}}(\x)$ it must hold either $\ones^\top \s_{\hat{\alpha}} = 1$ or $\ones^\top \s_{\hat{\alpha}} = \frac{1-\hat{\alpha}}{\hat{\alpha}} < 1$, and the continuous function $\ones^\top \s_{\hat{\alpha}}$ cannot jump from $1$ to $\frac{1-\hat{\alpha}}{\hat{\alpha}}$ without taking any intermediate value.
  
	The formula~\eqref{homotopy} now follows by Taylor expansion.
\end{proof}
This result suggests a strategy to improve convergence for $\alpha\approx 1$: we start solving the problem with a small value of $\alpha= \alpha_0$, e.g., $0.6$, then we solve it repeatedly for increasing values $\alpha_0 < \alpha_1 < \alpha_2 < \dots < \alpha_k = \alpha$; at each step $h$ we use as a starting point for our iterative algorithms an approximation of $\s_{\alpha_{h+1}}$ constructed from $\s_{\alpha_{h}}$.

We tried several different strategies to construct this approximation.
\begin{description}
	\item[T2] Second-order Taylor expansion $\s_{\alpha_{h+1}} \approx \s_{\alpha_{h}} + \frac{d \s_{\alpha}}{d\alpha}|_{\alpha_{h}} (\alpha_{h+1}-\alpha_{h}) + \frac12  \frac{d^2 \s_{\alpha}}{d\alpha^2}|_{\alpha_{h}} (\alpha_{h+1}-\alpha_{h})^2 $;
	\item[T1] First-order Taylor expansion $\s_{\alpha_{h+1}} \approx \s_{\alpha_{h}} + \frac{d \s_{\alpha}}{d\alpha}|_{\alpha_{h}} (\alpha_{h+1}-\alpha_{h})$;
	\item[EXT] Linear extrapolation from the previous two values $\s_{\alpha_{h+1}} \approx \s_{\alpha_h} + \frac{\s_{\alpha_h} - \s_{\alpha_{h-1}}}{\alpha_h - \alpha_{h-1}} (\alpha_{h+1}-\alpha_h)$;
	\item[IMP] An attempt at building a `semi-implicit method' with only the values available: ideally, one would like to set $\s_{\alpha_{h+1}} \approx \s_{\alpha_{h}} + \frac{d \s_{\alpha}}{d\alpha}|_{\alpha_{h+1}} (\alpha_{h+1}-\alpha_{h})$ (which would correspond to expressing $\s_{\alpha_h}$ with a Taylor expansion around $\alpha_{h+1}$); however, we do not know $\frac{d \s_{\alpha}}{d\alpha}|_{\alpha_{h+1}}$. Instead, we take $\s_{\alpha_{h+1}} \approx \s_{\alpha_{h}} + \s^{(1)} (\alpha_{h+1}-\alpha_{h})$ with
	\[
		\s^{(1)} = (I-\alpha_{h+1} R(\s_{\alpha_h} \otimes I) - \alpha_{h+1}R(I\otimes \s_{\alpha_h}))^{-1}R(\s_{\alpha_h} \otimes \s_{\alpha_h} - \v),
	\]
	i.e., we use the new value of $\alpha$ and the old value of $\s$ in the expression for the first derivative.
\end{description}

The only remaining choice is designing an effective strategy to choose the next point $\alpha_{h+1}$. We adopt the following one. Note that for \textsf{T1,EXT,IMP} we expect $\s_{\alpha_{h+1}} - \s_{\alpha_h} = O((\alpha_{h+1}-\alpha_h)^2)$. Hence one can expect
\[
	\frac{\s_{\alpha_h} - \s_{\alpha_{h-1}}}{(\alpha_h - \alpha_{h-1})^2} \approx \frac{\s_{\alpha_{h+1}} - \s_{\alpha_h}}{(\alpha_{h+1} - \alpha_h)^2}.
\]
Thus we choose a `target' value $\tau$ (e.g., $\tau=0.01$) for $\s_{\alpha_{h+1}} - \s_{\alpha_h}$, and then we choose $\alpha_{h+1}$ by solving the equation
\[
	\frac{\s_{\alpha_h} - \s_{\alpha_{h-1}}}{(\alpha_h - \alpha_{h-1})^2} = \frac{\tau}{(\alpha_{h+1} - \alpha_h)^2}.
\]
The resulting continuation algorithm is described in Algorithm~\ref{algo:homotopy}. Note that we start from $\alpha_0 = 0.6$, to steer clear of the double solution for $\alpha=0.5$. Convergence for such a value of $\alpha_0$ is typically not problematic.
\begin{algorithm}
\KwIn{$R$, $\alpha$, $\v$ as above, with $\alpha > 0.6$, a tolerance $\varepsilon$, a continuation `speed' parameter $\tau$.}
\KwOut{An approximation to a stochastic solution $\s$ of~\eqref{mlpr}.}
$\alpha_0 \leftarrow 0.6$\;
$h = 0$\;
\While{$\alpha_h < \alpha$}{
	\uIf{$h=0$}{$\x_{guess} \leftarrow \v$\;}
	\Else{
		$\x_{guess}$ determined using \textsf{T1,T2,IMP} or (if $h > 1$) \textsf{EXT}\; 
	}
	$\x_h \leftarrow \text{a stochastic solution to~\eqref{mlpr} with parameters $R$, $\alpha_h$, $\v$}$ (determined using a the Newton or Perron-Newton method, with initial value $\x_{guess}$)\;
	\uIf{$h=0$}{
		$\alpha_{h+1} \leftarrow \alpha_h + 0.01$\;
	}
	\Else{
		$\alpha_{h+1} = \alpha_h + \sqrt{\frac{\tau (\alpha_h - \alpha_{h-1})^2}{\x_h - \x_{h-1}}}$\;
	}
	\If{$\alpha_{h+1} > \alpha$}{
		$\alpha_{h+1} \leftarrow \alpha$\;
	}
	$h \leftarrow h+1$\;
}
$\s \leftarrow \x_{h-1}$\;

\caption{The continuation method for the computation of a stochastic solution $\s$ to~\eqref{mlpr}.} \label{algo:homotopy}
\end{algorithm}

Note that for strategy~\textsf{T2} one would expect $\s_{\alpha_{h+1}} - \s_{\alpha_{h}} = O((\alpha_{h+1}-\alpha_h)^3)$ instead, but for simplicity (and to ensure a fair comparison between the methods) we use the same step size selection strategy.

\section{Numerical experiments} \label{sec:experiments}
We have implemented the described methods using Matlab, and compared them to some of the algorithms in~\cite{mlpr_github} on the same benchmark set of small-size models ($n \in \{3,4,6\}$) used in~\cite{GleLY15}. We have used tolerance $\varepsilon=\sqrt{\mathbf{u}}$, where $\mathbf{u}$ is the machine precision, $\x_0=\v$ as an initial value unless differently specified, and  $\tau = 0.01$ for Algorithm~\ref{algo:homotopy}. To compute Perron vectors, we have used the output of Matlab's \texttt{eig}. For problems with larger $n$, different methods (\texttt{eigs}, inverse iteration, Gaussian elimination for kernel computation, etc.) can be considered~\cite{StewartBook}.

The results, for various values of $\alpha$, are reported in Tables~\ref{firsttable} to~\ref{lasttable}. Each of the 29 rows represents a different experiment in the benchmark set, and the columns stand for the different algorithms, according to the following legend.
\begin{description}
	\item[F] The fixed-point iteration~\eqref{fp}, from an initial value such that $\ones^\top \x_0 = 1$ (and with renormalization to $\ones^\top \x_k = 1$ at each step).
	\item[IO] The inner-outer iteration method, as described in~\cite{GleLY15}.
	\item[N] Newton's method with normalization to $\ones^\top \x_k = 1$ at each step, as described in~\cite{GleLY15}. Note that this is \emph{not} Algorithm~\ref{algo:newton}, which would converge instead to the minimal solution $\m$.
	\item[P] The Perron method (Algorithm~\ref{algo:perron}).
	\item[PN] The Perron--Newton method (Algorithm~\ref{algo:perronnewton}).
	\item[N-xxx, PN-xxx] \textsf{N} or \textsf{PN} methods combined with continuation with one of the four strategies described in Section~\ref{sec:homotopy}.
\end{description}
For each experiment, we report the number of iterations required, and the CPU times in seconds (obtained on Matlab R2017a on a computer equipped with a 64-bit Intel core i5-6200U processor). 
The value NaN inside a table represents failure to converge after 10,000 iterations.
For \textsf{P} and \textsf{PN}, the number of iterations is defined as the sum of the number of Newton iterations required to compute $\m$ with Algorithm~\ref{algo:newton} and the number of iterations in the main algorithm.
For the continuation algorithms, we report the number of inner iterations, that is, the total number of iterations performed by \textsf{PN} or \textsf{N} (respectively), summing along all calls to the subroutine.

In addition to the tables, performance profiles~\cite[Section~22.4]{matlabguide} are reported in Figures~\ref{firstfig}--\ref{lastfig}.

Note that neither the number of iterations nor the CPU time are particularly indicative of the true performance of the algorithms: indeed, the iterations in the different methods amount to different quantity of work, since some require merely linear system solutions and some require eigenvalue computations. Moreover, Matlab introduces overheads which are difficult to predict for several instructions (such as variable access and function calls). In order to make the time comparisons more fair, for methods \textsf{IO} and \textsf{N} we did not use the code in~\cite{mlpr_github}, but rather rewrote it without making use of Matlab's object-oriented features, which made the original functions significantly slower. The only change introduced with respect to their implementations is that we changed the stopping criterion to $\norm{G(\x)-\x}_1 \leq \sqrt{\mathbf{u}}$, so that the stopping criterion is the same one for all tested methods.

In any case, the performance comparison between the various methods should be taken with a pinch of salt. 

\begin{table}
\footnotesize
	\hspace*{-1.5cm}
	\pgfplotstableread{
	F IO N O ON ON-T1 ON-IMP ON-T2 ON-EXT N-T1 N-IMP N-T2 N-EXT
        184          45           4          19           7          30          23          23          47           8           6          23          13
          18          40           5          22           9          40          32          25          64          11           9          25          21
          18          40           5          14           8          40          32          25          64          11           9          25          21
          67          28           4          17           7          23          23          23          40           6           6          23          11
          54          83           6          29          10          62          30          26          94          19          13          26          30
         227         169           7          59           9          55          32          27          82          19          10          27          28
         680         142           5          63           8          39          25          25          57          11           7          25          17
         117         153           6          63           8          32          25          25          51           9           7          25          15
         512         166           5          63           8          37          25          25          51          11           7          25          15
         407         160           6          95           9          47          25          26          73          15           7          26          23
         654         140           5          97           8          39          24          25          58          12           6          25          17
         186         141           6          51           9          47          31          25          72          15           8          25          22
         NaN         118           6          73           8          46          31          25          64          15           9          25          21
         248         134           6          58           8          53          32          25          79          17          11          25          27
        2259         151           5          79           8          31          24          24          49          10           7          24          15
         171         186           5          92           7          31          23          23          42           9           6          23          12
         114         155           6          43           8          32          31          25          51          11          10          25          17
         240         175           6          68           8          47          25          26          72          16           8          26          24
         204         124           6          52           8          47          32          26          66          14          10          26          22
         594         140           5          57           8          31          24          25          50           9           6          25          14
         227         167           5          79           8          47          25          26          66          14           7          26          21
        1514         207           8         103           7          31          25          24          50           9           8          24          16
         119         127           8          55          10          40          27          28          59          14          10          28          20
         164         109           5          51           7          40          24          25          58          12           7          25          17
         170         225           7          46           8          48          33          25          65          15          11          25          23
          91         110           6          40           9          48          26          26          71          14           8          26          22
          69          99           5          38           7          31          24          24          43           8           6          24          12
         112         131           6          58           8          38          31          25          65          13          10          25          22
          70         106           6          38           7          30          24          24          48          10           8          24          16	
    }\tableiter
    \pgfplotstabletypeset[]{\tableiter}

    	\hspace*{-2.5cm}
	\pgfplotstabletypeset[precision=1, sci, sci zerofill, sci e]{
	F IO N O ON ON-T1 ON-IMP ON-T2 ON-EXT N-T1 N-IMP N-T2 N-EXT
   4.1430e-03   2.3124e-02   3.1600e-04   1.4570e-03   7.1100e-04   2.4870e-03   1.8250e-03   2.3760e-03   3.5890e-03   9.0900e-04   7.1000e-04   2.1980e-03   1.0750e-03
   4.8900e-04   1.6079e-02   3.4100e-04   1.6300e-03   8.6600e-04   3.2720e-03   2.9670e-03   2.5840e-03   8.4370e-03   1.1630e-03   9.5600e-04   2.6190e-03   1.5330e-03
   4.6800e-04   2.0442e-02   5.2300e-04   1.8660e-03   1.0650e-03   5.3260e-03   4.3690e-03   2.5000e-03   5.4210e-03   1.2250e-03   9.9100e-04   2.5710e-03   1.6180e-03
   1.5270e-03   8.9000e-03   3.0400e-04   1.3550e-03   7.5100e-04   2.1940e-03   2.0440e-03   2.2860e-03   5.9020e-03   1.2560e-03   1.2500e-03   4.0870e-03   1.6740e-03
   2.2290e-03   8.0691e-02   6.5800e-04   3.8530e-03   2.0370e-03   1.4526e-02   4.5090e-03   3.5780e-03   1.6009e-02   4.0030e-03   2.8560e-03   8.1090e-03   4.6520e-03
   1.2113e-02   1.0295e-01   4.3900e-04   4.7160e-03   1.0280e-03   4.6340e-03   2.6270e-03   2.5770e-03   6.3780e-03   1.4740e-03   8.1100e-04   2.5660e-03   1.7810e-03
   1.3188e-02   7.1900e-02   3.2300e-04   4.9350e-03   8.5000e-04   3.0520e-03   2.0470e-03   2.3550e-03   4.3190e-03   9.6100e-04   5.8800e-04   2.2770e-03   1.0930e-03
   2.2780e-03   7.1425e-02   3.6400e-04   4.5110e-03   8.2200e-04   2.5570e-03   2.0030e-03   2.2680e-03   3.8110e-03   7.5700e-04   5.6800e-04   2.2540e-03   9.8800e-04
   1.1389e-02   8.3916e-02   3.2400e-04   4.8930e-03   8.5000e-04   3.0140e-03   1.9870e-03   2.2690e-03   3.8330e-03   9.1100e-04   5.8300e-04   2.2680e-03   9.6700e-04
   7.9300e-03   8.5992e-02   7.4000e-04   1.0386e-02   1.8230e-03   5.9140e-03   2.6310e-03   3.0200e-03   7.7510e-03   1.4130e-03   6.5500e-04   2.7420e-03   2.1180e-03
   1.6825e-02   7.3235e-02   5.3000e-04   7.4790e-03   8.7200e-04   3.1820e-03   1.9160e-03   2.2130e-03   4.3830e-03   9.8800e-04   5.2900e-04   2.2310e-03   1.1110e-03
   3.6270e-03   7.1445e-02   3.6400e-04   4.2010e-03   9.6900e-04   3.7780e-03   2.4250e-03   2.1510e-03   5.2910e-03   1.1700e-03   7.3200e-04   2.1640e-03   1.4210e-03
          NaN   8.8558e-02   4.4900e-04   5.7500e-03   8.7400e-04   3.7100e-03   2.5080e-03   2.1790e-03   4.7420e-03   1.2980e-03   7.3000e-04   2.1990e-03   1.3080e-03
   4.8520e-03   7.0618e-02   3.7000e-04   4.5670e-03   8.1800e-04   4.1310e-03   2.5150e-03   2.1500e-03   5.9220e-03   1.3410e-03   8.1700e-04   2.1520e-03   1.6510e-03
   4.2763e-02   7.3885e-02   3.2700e-04   5.6400e-03   8.0900e-04   2.5010e-03   1.8660e-03   1.9990e-03   3.5060e-03   7.8600e-04   5.5300e-04   1.9970e-03   1.0090e-03
   3.3920e-03   8.1935e-02   3.3300e-04   6.4920e-03   6.6100e-04   2.4000e-03   1.7330e-03   1.8700e-03   3.0620e-03   7.4400e-04   5.7300e-04   1.8930e-03   8.1200e-04
   2.2960e-03   7.4605e-02   3.7100e-04   3.0620e-03   7.8900e-04   2.6100e-03   2.3530e-03   2.1830e-03   3.7580e-03   8.3100e-04   7.7100e-04   2.1220e-03   1.1280e-03
   4.7090e-03   9.1475e-02   3.5900e-04   5.3720e-03   8.2300e-04   3.7300e-03   2.0030e-03   2.2810e-03   5.4130e-03   1.2440e-03   6.0700e-04   2.2670e-03   1.4900e-03
   3.9760e-03   6.3811e-02   3.6600e-04   4.0120e-03   8.5500e-04   3.7710e-03   2.8350e-03   2.6650e-03   5.5010e-03   1.2260e-03   8.8200e-04   2.5840e-03   1.4910e-03
   1.1390e-02   6.8946e-02   3.3200e-04   4.4440e-03   8.2000e-04   2.5630e-03   1.8300e-03   2.1870e-03   3.7010e-03   8.2700e-04   6.4700e-04   4.2040e-03   1.0590e-03
   4.4070e-03   8.6158e-02   3.2300e-04   6.0440e-03   8.1600e-04   3.8530e-03   1.9670e-03   2.2650e-03   4.9280e-03   1.1140e-03   6.0500e-04   2.2590e-03   1.3420e-03
   2.8808e-02   8.8135e-02   4.5800e-04   7.2060e-03   6.9300e-04   2.3580e-03   1.9010e-03   1.9990e-03   3.6280e-03   7.3900e-04   6.0800e-04   1.9680e-03   1.0050e-03
   2.3430e-03   6.8446e-02   6.1000e-04   4.2160e-03   1.2310e-03   3.6220e-03   2.3800e-03   2.7460e-03   4.9620e-03   1.2190e-03   8.2700e-04   2.8010e-03   1.3240e-03
   3.3890e-03   5.6047e-02   4.6300e-04   4.0400e-03   6.6400e-04   3.2880e-03   1.8520e-03   2.1780e-03   4.3430e-03   9.5700e-04   5.5700e-04   2.2110e-03   1.0910e-03
   3.5300e-03   1.1245e-01   4.4700e-04   4.0390e-03   9.0800e-04   4.3250e-03   2.9400e-03   2.4910e-03   5.4990e-03   1.2890e-03   8.9400e-04   2.3880e-03   1.5340e-03
   1.9290e-03   5.8293e-02   4.2800e-04   3.4990e-03   1.0340e-03   4.4550e-03   2.4130e-03   2.6310e-03   6.2490e-03   1.2840e-03   6.7900e-04   2.6640e-03   1.5860e-03
   1.5310e-03   4.4349e-02   3.3600e-04   3.0160e-03   7.0900e-04   2.6290e-03   2.0030e-03   2.1960e-03   3.4830e-03   7.7300e-04   5.5600e-04   2.1900e-03   8.7500e-04
   2.4850e-03   6.6492e-02   3.8700e-04   5.0220e-03   9.5200e-04   3.3170e-03   2.7240e-03   2.3870e-03   5.3880e-03   1.0990e-03   8.6700e-04   2.3860e-03   1.5300e-03
   1.5310e-03   4.8125e-02   4.0500e-04   3.3280e-03   7.4800e-04   2.6220e-03   2.0340e-03   2.2840e-03   3.8070e-03   8.6300e-04   7.4900e-04   2.3320e-03   1.1100e-03
}	
	\caption{Iteration counts and times for the 29 benchmark tensors and $\alpha = 0.90$} \label{firsttable}
\end{table}

\begin{table}
\footnotesize
	\hspace*{-1.5cm}
	\pgfplotstabletypeset[]{
	F IO N O ON ON-T1 ON-IMP ON-T2 ON-EXT N-T1 N-IMP N-T2 N-EXT
         NaN          42           5          21           6          35          22          22          46          10           6          22          13
          23          48           6          26           7          50          36          24          80          17          13          24          29
          23          48           6          10           7          50          36          24          80          17          13          24          29
         114          24           4          19           6          28          22          22          39           8           6          22          11
          22          41          17          15           7          74          26          24         110          23          29          24          37
         NaN         639           7         NaN         NaN          73          39          25         109          26          14          25          38
         NaN         234           5         369           8          44          24          24          63          13           7          24          20
         NaN        1221           6        1845           9          46          26          26          69          16           9          26          24
         NaN         532           5         NaN           7          44          24          25          63          15           7          25          20
         NaN         339           6         NaN           7          59          30          25          85          20           9          25          29
         NaN         228           5         NaN           7          45          23          25          64          15           7          25          20
         NaN         357           6         362         NaN          62          31          26          90          21           9          26          30
         NaN         156           6        7509           6          52          31          25          77          18          10          25          27
         NaN         329           6         NaN         NaN          65          32          25          92          22          11          25          33
         NaN         550           6         NaN           8          38          23          25          56          13           7          25          19
         NaN         NaN           5         NaN           7          31          23          23          47          10           7          23          14
         NaN        1173         NaN         NaN          71          99          32         NaN          87         NaN          11         NaN         NaN
         NaN         541           6        1595         NaN          60          31          25          89          21          10          25          32
         NaN         273         NaN         248           8          53          37         NaN          78          17          12         NaN          27
         NaN         230           7         196           7          37          24          24          56          12           7          24          17
         NaN         533           6         NaN           9          59          31          26          90          19          10          26          29
         NaN         NaN           7         205           6          31          32          24          50          10          16          24          17
         NaN         435         NaN         504         NaN          59          35         NaN          89          23         NaN         NaN          31
         NaN         147           5         137           7          46          29          25          64          15           8          25          20
         NaN         835         NaN         NaN          44          76          59          41         121          40          36          41          44
         NaN         654         NaN         NaN         NaN          72          33         NaN         103          24          12         NaN          37
         NaN         NaN         NaN        8025        1326         NaN         903         652        2878         NaN         NaN         652         NaN
         NaN         334           6         307         NaN          57          37          24          91          21          13          24          33
        2005        2100           9         909          10          34          33          27          52          14          14          27          20
    }

    	\hspace*{-2.5cm}
	\pgfplotstabletypeset[precision=1, sci, sci zerofill, sci e]{
	F IO N O ON ON-T1 ON-IMP ON-T2 ON-EXT N-T1 N-IMP N-T2 N-EXT
          NaN   4.0287e-02   6.8370e-03   8.0380e-03   4.7460e-03   1.0794e-02   4.8260e-03   5.8560e-03   6.3080e-03   2.0280e-03   8.1200e-04   2.2610e-03   1.1850e-03
   8.8000e-04   2.2916e-02   4.5600e-04   2.1200e-03   1.0220e-03   4.6310e-03   3.7680e-03   2.5030e-03   6.6260e-03   1.5380e-03   1.1940e-03   2.3830e-03   2.0210e-03
   5.5100e-04   2.2017e-02   3.7100e-04   9.2300e-04   7.7400e-04   4.1440e-03   3.2120e-03   2.5020e-03   6.8380e-03   1.5720e-03   1.2460e-03   2.6200e-03   2.0110e-03
   2.4200e-03   7.7390e-03   2.8800e-04   1.4320e-03   8.9500e-04   2.7390e-03   1.8910e-03   2.1230e-03   3.0710e-03   8.7000e-04   6.6200e-04   2.1360e-03   9.9600e-04
   5.3400e-04   2.4004e-02   9.7100e-04   1.2020e-03   9.8300e-04   7.4120e-03   3.0140e-03   3.0130e-03   9.3460e-03   2.5620e-03   2.2510e-03   3.0420e-03   3.1020e-03
          NaN   3.8149e-01   4.1900e-04          NaN          NaN   1.2314e-02   4.2050e-03   7.0560e-03   1.5590e-02   2.0580e-03   1.0770e-03   2.3690e-03   2.4360e-03
          NaN   1.6121e-01   3.4400e-04   3.0188e-02   1.1910e-03   3.8840e-03   2.1670e-03   2.1970e-03   4.9730e-03   1.1040e-03   5.5900e-04   2.1880e-03   1.2850e-03
          NaN   7.3718e-01   3.6500e-04   1.4164e-01   1.0710e-03   3.9600e-03   2.1840e-03   2.3380e-03   5.4740e-03   1.2160e-03   6.4300e-04   2.3670e-03   1.4730e-03
          NaN   3.3547e-01   3.3400e-04          NaN   1.5340e-03   8.3990e-03   5.0270e-03   5.5720e-03   1.1868e-02   1.2610e-03   6.1400e-04   6.2920e-03   7.0260e-03
          NaN   2.1960e-01   3.6500e-04          NaN   1.2700e-03   8.0300e-03   4.1330e-03   2.9080e-03   1.2365e-02   2.0630e-03   9.6200e-04   3.9000e-03   2.3800e-03
          NaN   1.5050e-01   3.5900e-04          NaN   1.3520e-03   8.9850e-03   4.6960e-03   6.5830e-03   9.0070e-03   2.7150e-03   1.6230e-03   8.1840e-03   1.4770e-03
          NaN   2.3649e-01   3.7200e-04   2.8417e-02          NaN   1.1532e-02   6.6930e-03   5.3660e-03   9.8210e-03   3.1890e-03   1.0040e-03   2.7400e-03   1.9670e-03
          NaN   1.1705e-01   4.1200e-04   5.8156e-01   6.4700e-04   4.2370e-03   2.5050e-03   2.3390e-03   6.0010e-03   1.3890e-03   7.7100e-04   2.2390e-03   1.6530e-03
          NaN   2.3280e-01   3.9500e-04          NaN          NaN   1.0838e-02   3.5880e-03   2.5240e-03   1.1843e-02   4.1770e-03   1.3670e-03   4.4840e-03   5.5170e-03
          NaN   3.2835e-01   5.5400e-04          NaN   1.5180e-03   7.5060e-03   4.5960e-03   5.8860e-03   8.2110e-03   1.9440e-03   1.0940e-03   7.1670e-03   3.6620e-03
          NaN          NaN   3.6100e-04          NaN   1.5610e-03   6.3950e-03   3.9520e-03   5.1360e-03   8.2510e-03   1.5320e-03   9.5700e-04   4.1400e-03   2.3140e-03
          NaN   6.5247e-01          NaN          NaN   1.8720e-02   2.0032e-02   3.1680e-03          NaN   1.4657e-02          NaN   1.6350e-03          NaN          NaN
          NaN   3.3779e-01   3.7100e-04   1.2345e-01          NaN   1.0944e-02   3.5340e-03   3.0530e-03   1.2610e-02   1.7490e-03   8.0600e-04   3.2700e-03   2.4930e-03
          NaN   1.8601e-01          NaN   3.0119e-02   9.3800e-04   7.6060e-03   3.2210e-03          NaN   1.5140e-02   3.1310e-03   2.5710e-03          NaN   4.2390e-03
          NaN   1.6123e-01   4.0800e-04   1.5513e-02   7.6000e-04   3.1430e-03   1.9770e-03   2.1730e-03   4.5010e-03   9.6300e-04   5.6600e-04   2.1010e-03   1.1970e-03
          NaN   3.4548e-01   3.7100e-04          NaN   1.6350e-03   1.0167e-02   5.9270e-03   6.4150e-03   9.5840e-03   6.5200e-03   2.4370e-03   5.1280e-03   2.1240e-03
          NaN          NaN   4.2300e-04   1.4687e-02   6.0500e-04   2.5420e-03   2.6630e-03   2.0480e-03   3.7130e-03   7.9000e-04   1.0700e-03   2.0900e-03   1.0450e-03
          NaN   3.0682e-01          NaN   5.6647e-02          NaN   1.2619e-02   9.5850e-03          NaN   1.7569e-02   4.4860e-03          NaN          NaN   4.1480e-03
          NaN   1.0853e-01   3.3100e-04   1.0646e-02   7.6000e-04   3.8090e-03   2.3720e-03   2.2210e-03   5.0210e-03   1.1820e-03   6.9600e-04   2.2760e-03   1.2590e-03
          NaN   5.3200e-01          NaN          NaN   1.4329e-02   1.6637e-02   1.3369e-02   4.9140e-03   1.3508e-02   3.0080e-03   4.5580e-03   4.9680e-03   3.0070e-03
          NaN   4.3354e-01          NaN          NaN          NaN   1.3447e-02   3.8760e-03          NaN   1.5697e-02   2.4050e-03   1.0470e-03          NaN   6.0390e-03
          NaN          NaN          NaN   6.8039e-01   1.7791e-01          NaN   1.5141e-01   8.7778e-02   3.8609e-01          NaN          NaN   1.1820e-01          NaN
          NaN   2.2634e-01   4.0400e-04   2.6544e-02          NaN   1.3434e-02   9.3020e-03   4.1000e-03   1.6384e-02   2.0090e-03   1.1420e-03   2.4780e-03   2.9840e-03
   4.2918e-02   1.0304e+00   5.4600e-04   7.9244e-02   1.2830e-03   3.2800e-03   3.3320e-03   2.8040e-03   4.6310e-03   1.0830e-03   1.0800e-03   2.8020e-03   1.2980e-03
}
	\caption{Iteration counts and times for the 29 benchmark tensors and $\alpha = 0.99$}
\end{table}

\begin{table}
\footnotesize
	\hspace*{-1.5cm}
	\pgfplotstabletypeset[]{
	F IO N O ON ON-T1 ON-IMP ON-T2 ON-EXT N-T1 N-IMP N-T2 N-EXT
         NaN          41           5          21           6          35          22          22          46          10           6          22          13
          29          51           6          28           7          50          36          24          80          18          13          24          29
          29          51           6           9           6          50          36          24          80          18          13          24          29
         122          24           4          19           6          28          22          22          39           8           6          22          11
          18          38         336          13           8          74          26          23         111          23          71          23          37
         NaN         NaN           7         NaN         NaN          73          39          25         110          26          14          25          39
         NaN         251           5         723           8          45          24          25          63          14           7          25          20
         NaN         NaN           6         NaN           9          46          26          26          69          16           9          26          24
         NaN         NaN           5         NaN           7          44          24          25          63          15           7          25          20
         NaN         383           6         NaN           7          59          31          25          85          20           9          25          29
         NaN         243           5         NaN           7          45          23          25          64          15           7          25          20
         NaN         442           6         824         NaN          63          31          26          95          22           9          26          32
         NaN         164           6         NaN           6          52          31          25          77          18          10          25          27
         NaN         396           6         NaN         NaN          66          36          26          97          23          12          26          35
         NaN         812           6         NaN           8          38          23          25          60          13           7          25          21
         NaN         NaN           5         NaN           7          31          23          23          47          10           7          23          14
         NaN         NaN         NaN         NaN         NaN         100          33         NaN          92         NaN          12         NaN         NaN
         NaN         848           6         NaN         NaN          65          31          25          93          23          10          25          33
         NaN         323         NaN         388           8          54          37         NaN          78          18          12         NaN          27
         NaN         253           7         258           7          37          24          24          56          12           7          24          17
         NaN         857           6         NaN           9          59          31          26          90          19          10          26          29
         NaN         334           7         165           6          30          32          24          50          10          16          24          17
         NaN         577         NaN        1185         NaN          59          35         NaN          94          23         NaN         NaN          33
         NaN         150           5         160           7          46          29          25          69          15           8          25          22
         NaN         NaN         NaN         NaN         NaN          81          60         NaN         126          42          37         NaN          46
         NaN         NaN         NaN         NaN         NaN          72          38         NaN         108          25          14         NaN          39
         348         358          50         276         192         108          64          53          71         232          54          53         131
         NaN         406           7         522         NaN          57          38          24          91          21          13          24          33
         205         224          12         108         NaN          48         NaN          30          56          31          17          30         NaN
         }

    \hspace*{-2.5cm}
	\pgfplotstabletypeset[precision=1, sci, sci zerofill, sci e]{
	F IO N O ON ON-T1 ON-IMP ON-T2 ON-EXT N-T1 N-IMP N-T2 N-EXT
          NaN   3.1291e-02   3.6400e-04   1.6240e-03   6.8900e-04   3.4260e-03   2.0030e-03   6.5270e-03   7.4010e-03   1.2600e-03   8.4500e-04   5.3210e-03   1.3680e-03
   8.8600e-04   2.7697e-02   4.0300e-04   2.1740e-03   7.6200e-04   4.4250e-03   3.1470e-03   2.4460e-03   6.4380e-03   1.5740e-03   1.1270e-03   2.3110e-03   1.9420e-03
   6.6800e-04   2.4392e-02   3.6500e-04   7.3400e-04   6.1200e-04   4.4820e-03   3.1190e-03   2.3090e-03   6.3190e-03   1.6150e-03   1.1530e-03   2.3540e-03   2.0990e-03
   2.5570e-03   7.5000e-03   2.7800e-04   1.4020e-03   6.3800e-04   2.5050e-03   1.9170e-03   2.1400e-03   3.0620e-03   9.4400e-04   6.6500e-04   2.0110e-03   9.1800e-04
   4.2800e-04   2.1901e-02   1.4668e-02   1.0140e-03   8.8100e-04   8.8560e-03   6.4340e-03   4.5580e-03   1.7396e-02   2.9230e-03   1.0287e-02   3.1630e-03   3.2360e-03
          NaN          NaN   4.1100e-04          NaN          NaN   1.1972e-02   4.4210e-03   6.6630e-03   1.1674e-02   4.3250e-03   1.1010e-03   2.6410e-03   2.5130e-03
          NaN   1.7109e-01   3.2300e-04   5.5186e-02   8.9700e-04   3.6650e-03   2.0110e-03   2.2640e-03   4.9540e-03   1.1660e-03   5.5700e-04   2.2230e-03   1.2490e-03
          NaN          NaN   3.6000e-04          NaN   1.9000e-03   7.9650e-03   6.0360e-03   7.9450e-03   1.6494e-02   2.4700e-03   1.2610e-03   4.8860e-03   2.9180e-03
          NaN          NaN   3.2300e-04          NaN   1.4060e-03   6.4940e-03   3.8950e-03   3.7000e-03   9.0870e-03   2.6350e-03   1.3920e-03   5.7580e-03   3.6680e-03
          NaN   2.6118e-01   3.9000e-04          NaN   1.2920e-03   9.4700e-03   7.2770e-03   7.4980e-03   1.2546e-02   2.6710e-03   8.2700e-04   2.3950e-03   1.8460e-03
          NaN   1.6603e-01   3.2300e-04          NaN   1.3480e-03   8.0580e-03   2.8910e-03   2.6810e-03   8.1370e-03   1.3860e-03   7.2200e-04   4.7140e-03   1.7630e-03
          NaN   2.9908e-01   3.6500e-04   6.4196e-02          NaN   1.1143e-02   3.1900e-03   2.7020e-03   1.5662e-02   1.8940e-03   7.7100e-04   2.5120e-03   3.4960e-03
          NaN   1.2401e-01   3.6600e-04          NaN   1.5350e-03   1.0045e-02   5.3590e-03   7.8010e-03   1.4924e-02   1.5740e-03   8.5700e-04   2.3700e-03   1.6960e-03
          NaN   2.6504e-01   3.6400e-04          NaN          NaN   1.1219e-02   3.7010e-03   2.7640e-03   1.2146e-02   2.0010e-03   9.9300e-04   3.2320e-03   2.5750e-03
          NaN   4.4539e-01   3.6300e-04          NaN   1.3680e-03   5.8260e-03   3.9740e-03   3.7160e-03   9.1870e-03   1.8230e-03   1.2140e-03   5.7720e-03   1.9310e-03
          NaN          NaN   3.2600e-04          NaN   1.2470e-03   7.5570e-03   3.1430e-03   2.5950e-03   3.9120e-03   8.6400e-04   6.2200e-04   2.3910e-03   1.8650e-03
          NaN          NaN          NaN          NaN          NaN   1.6500e-02   4.0970e-03          NaN   1.4849e-02          NaN   1.7490e-03          NaN          NaN
          NaN   4.6695e-01   4.2200e-04          NaN          NaN   1.3317e-02   5.5810e-03   5.7550e-03   1.0798e-02   3.2660e-03   1.0910e-03   2.4840e-03   2.5040e-03
          NaN   2.3283e-01          NaN   4.2788e-02   9.4500e-04   4.6330e-03   3.1550e-03          NaN   1.0924e-02   2.0740e-03   1.2150e-03          NaN   4.2410e-03
          NaN   1.6931e-01   4.6400e-04   2.0199e-02   7.5000e-04   3.0950e-03   1.9160e-03   2.0720e-03   4.2530e-03   9.6800e-04   5.5900e-04   2.0760e-03   1.0690e-03
          NaN   4.8344e-01   3.6400e-04          NaN   2.0670e-03   1.1857e-02   6.1740e-03   6.3050e-03   1.2079e-02   1.9610e-03   8.4700e-04   2.6720e-03   1.9370e-03
          NaN   2.2834e-01   4.0500e-04   1.1780e-02   5.9300e-04   2.3510e-03   2.7110e-03   2.0480e-03   3.6660e-03   8.7000e-04   1.1150e-03   2.0730e-03   1.0460e-03
          NaN   3.9349e-01          NaN   1.0941e-01          NaN   8.8260e-03   6.0830e-03          NaN   1.8621e-02   5.2300e-03          NaN          NaN   6.0740e-03
          NaN   1.0547e-01   3.6000e-04   1.2294e-02   7.4600e-04   3.8130e-03   2.3010e-03   2.2640e-03   5.3960e-03   1.2210e-03   6.8900e-04   2.2640e-03   1.3740e-03
          NaN          NaN          NaN          NaN          NaN   2.0796e-02   8.1500e-03          NaN   2.2634e-02   5.5930e-03   3.3770e-03          NaN   7.1310e-03
          NaN          NaN          NaN          NaN          NaN   1.6879e-02   9.8280e-03          NaN   1.9226e-02   6.7050e-03   2.9170e-03          NaN   6.6630e-03
   1.1303e-02   2.5756e-01   2.4790e-03   2.4581e-02   2.6294e-02   1.5118e-02   8.4650e-03   6.9650e-03   7.9730e-03   1.2323e-02   2.9270e-03   6.4780e-03   6.8280e-03
          NaN   2.7027e-01   4.5400e-04   4.5604e-02          NaN   1.1540e-02   5.4990e-03   5.6020e-03   1.4193e-02   2.0920e-03   1.1640e-03   2.5340e-03   2.2820e-03
   5.3300e-03   1.8944e-01   6.8300e-04   9.6340e-03          NaN   9.1600e-03          NaN   8.7030e-03   1.0370e-02   5.7700e-03   1.3530e-03   4.5030e-03          NaN
   }
	\caption{Iteration counts and times for the 29 benchmark tensors and $\alpha = 0.999$} \label{lasttable}
\end{table}

\begin{figure}
	\centering
	\includegraphics[width=\textwidth]{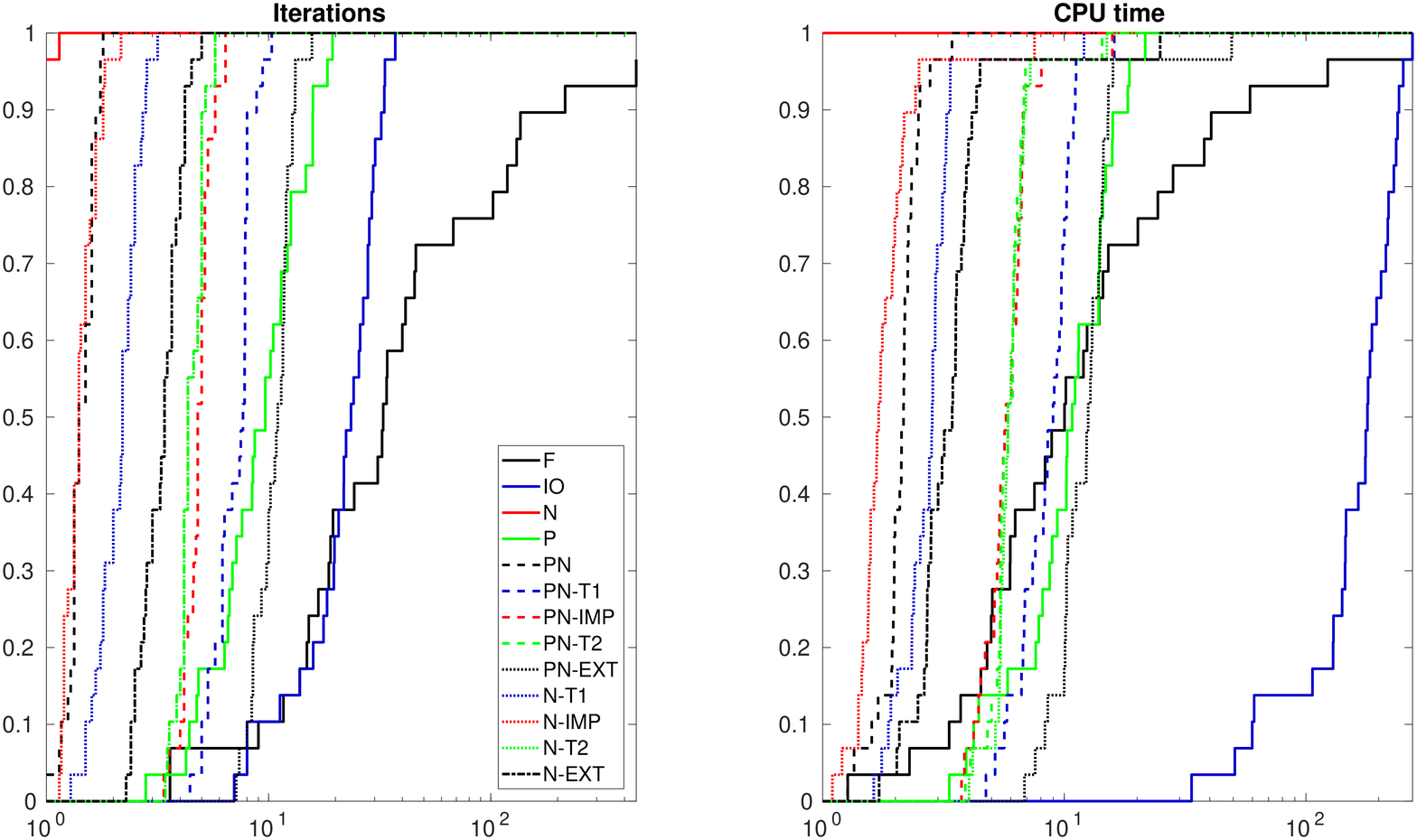}
	\caption{Performance profiles for the 29 benchmark tensors and $\alpha=0.9$} \label{firstfig}
\end{figure}
\begin{figure}
	\centering
	\includegraphics[width=\textwidth]{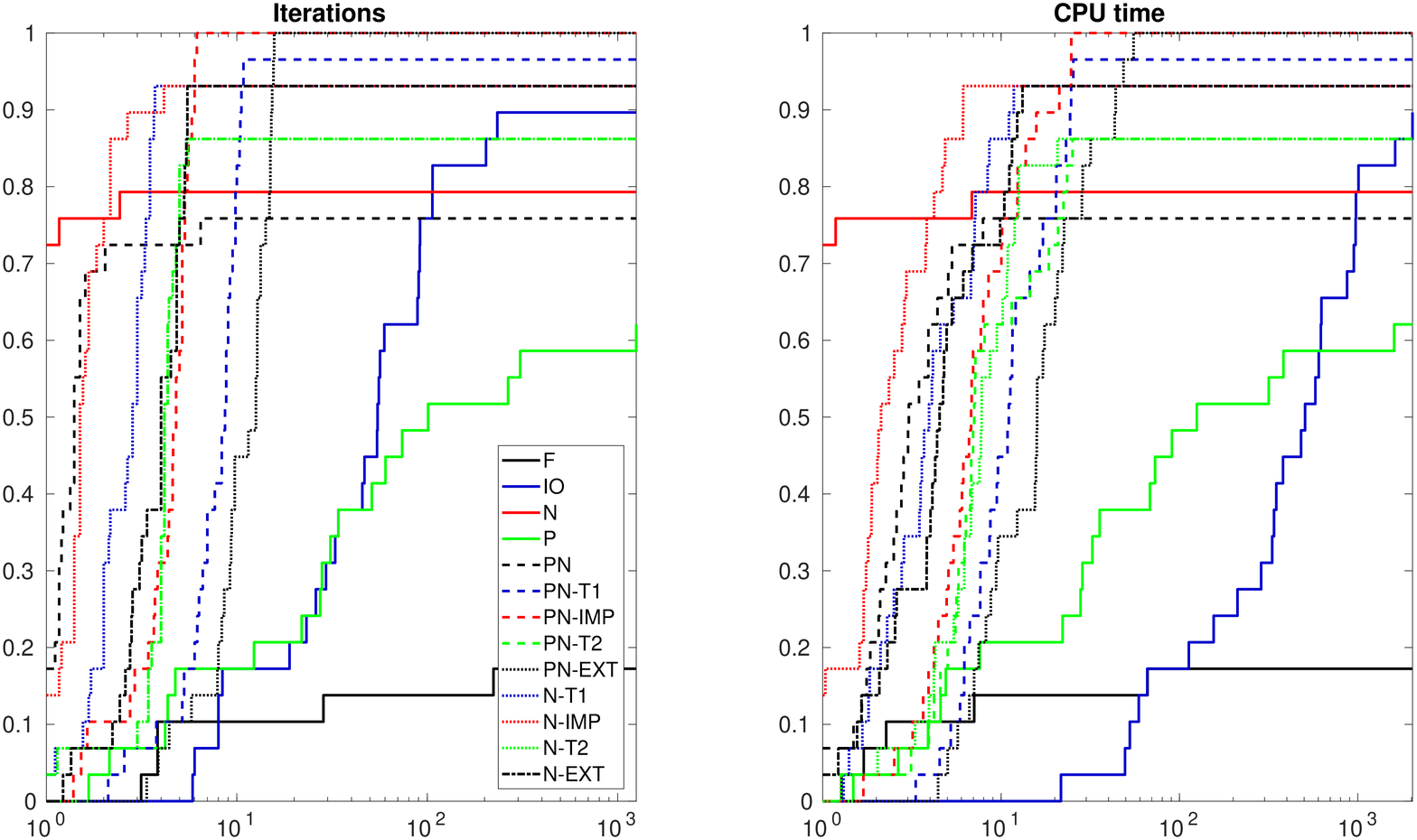}
	\caption{Performance profiles for the 29 benchmark tensors and $\alpha=0.99$}
\end{figure}
\begin{figure}
	\centering
	\includegraphics[width=\textwidth]{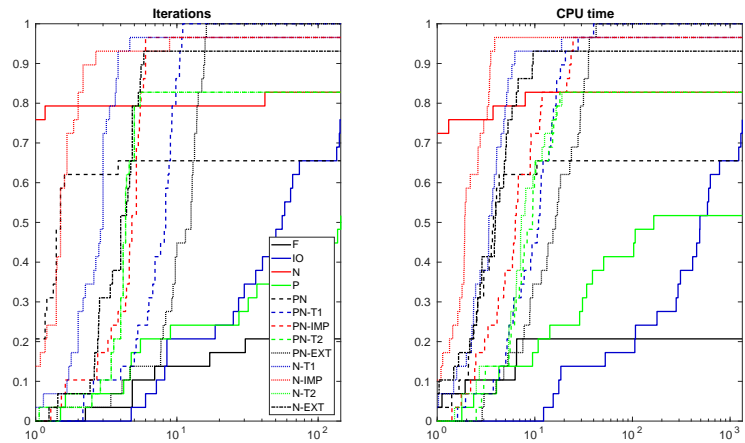}
	\caption{Performance profiles for the 29 benchmark tensors and $\alpha=0.999$} \label{lastfig}
\end{figure}

We comment briefly on the results. Newton-based methods typically require a constant number of iterations to converge, but on some of the benchmarks they fail to converge or require a much larger number of iterations. From the point of view of reliability, combining the Perron--Newton algorithm with continuation strategies gives a definite advantage: the resulting methods are the only ones (among the ones considered) that can solve successfully all the problems in the original benchmark set in~\cite{GleLY15}, which contained the experiments with all values of $\alpha$ up to $0.99$. Raising $\alpha$ to the more extreme value of $0.999$ reveals failure cases for these methods as well. (However, if one reduces the step-size selection parameter $\tau$ to $0.001$, methods~\textsf{PN-EXT} and~\textsf{PN-T1} succeed on all examples also for $\alpha=0.999$.)

Our results confirm the findings of~\cite{GleLY15} that problem \texttt{R6\_3} (the third to last one) for $\alpha=0.99$ is the hardest problem of the benchmark set; Perron-based algorithms can solve it successfully in the end, but (like the algorithms in~\cite{GleLY15}) they stagnate for a large number of iterations around a point which is far away from the true solution. The analysis in~\cite[End of Section~6.3 and caption of Figure~9]{GleLY15} also attributes this difficulty to the presence of a `pseudo-solution' that is an attraction basin for a large set of starting points.

To understand the specific difficulties encountered by continuation methods on problem \texttt{R6\_3}, it is instructive to consider the plot in Figure~\ref{fig:homoplot}.
\begin{figure}
\centering
\includegraphics[width=\textwidth]{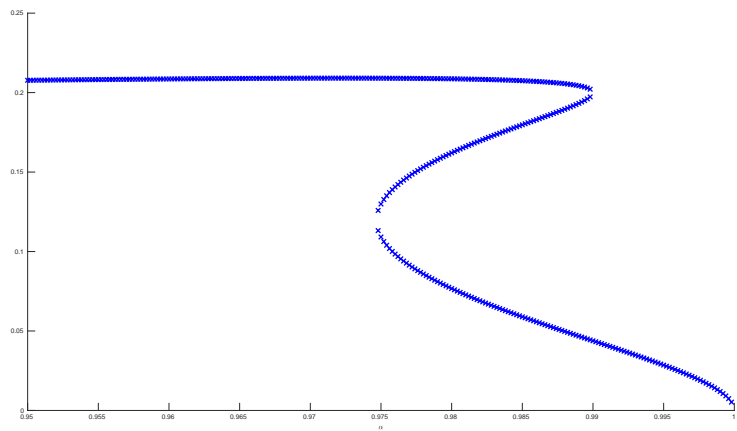}

\caption{The first entry $\s_1$ of each stochastic solution to~\eqref{mlpr} for the benchmark problem \texttt{R6\_3}, plotted for various values of $\alpha$ (on the x-axis). The behavior of all other entries of $\s$ is similar. } \label{fig:homoplot}
\end{figure}
This figure, which has been generated with the help of the software Bertini~\cite{Bertini}, shows the first entry of each valid stochastic solution $\s$ to~\eqref{mlpr} (recall that there may be more than one solution), for varying values of $\alpha$. It turns out that the problem has a single solution for all values of $\alpha$ up to $\alpha \approx 0.975$, then two more solutions appear, initially coincident and then slowly separating one from the other. Two of these three solutions merge at about $\alpha \approx 0.9899$ and then disappear (i.e., they are replaced by a pair of complex conjugate solutions), so for $\alpha=0.99$ there is only one real stochastic solution remaining. Note that the points where the number of solutions changes are points in which the Jacobian $I-G'_{\x}$ is singular, in accordance with the implicit function theorem. 

Continuation methods `track' the upper solution, and hence fail to provide a good starting point for $\alpha = 0.99$, when it disappears suddenly and is replaced by a completely different solution. Moreover, for $\alpha = 0.99$, the numerical behavior of these iterations is affected by the presence of a point with a very small residual. For an intuition, the reader may consider the behavior of Newton's method (on a function of a single variable) in presence of a local minimum very close to $0$, e.g., on $f(x)=x^2+\varepsilon$ for a very small $\varepsilon$.

\section{Conclusions} \label{sec:conclusions}
We have used the theory of quadratic vector equations in~\cite{BinMP11,MeiP11,Pol13} to attack the multilinear PageRank problem described in~\cite{GleLY15}. The tools that we have introduced improve the theoretical understanding of this equation and of its solutions. In particular, considering the minimal solution in addition to the stochastic ones allows one to use a broader array of algorithms, with computational advantage. The new algorithms achieve better results when $\alpha \approx 1$, which is the most interesting and computationally challenging case.

Possible future improvements to these algorithms include experimenting with further step-size selection strategies, and more importantly integrating them with inexact eigenvalue solvers that are more suitable to computation with large sparse matrices.

%\printbibliography

\bibliography{mlpr}
\end{document}